\documentclass[a4paper,11pt, reqno]{amsart}
\usepackage[all]{xy}
\usepackage{stmaryrd}
\usepackage{enumerate}
\usepackage{amssymb}
\usepackage[dvips]{color}

\usepackage[pdftex]{graphicx}\pdfcompresslevel=9

\newcommand{\comment}[1]{}

\newtheorem{thm}{Theorem}[section]
\newtheorem{prop}[thm]{Proposition}
\newtheorem{lem}[thm]{Lemma}
\newtheorem{cor}[thm]{Corollary}

\theoremstyle{definition}
\newtheorem{defn}[thm]{Definition}
\newtheorem{ex}[thm]{Example}

\theoremstyle{remark}
\newtheorem{rem}[thm]{Remark}

\newcommand{\Rr}{\mathbb R} 										
\newcommand{\set}[1]{\left\{#1\right\}}					

\newcommand{\eps}{\varepsilon}									
\newcommand{\tto}{\rightrightarrows}						

\newcommand{\F}{\ensuremath{\mathcal{F}}}

\newcommand{\NN}{\ensuremath{\mathcal{N}}}

\newcommand{\red}{\ensuremath{/\mkern-5mu/}}

\renewcommand{\graph}{\operatorname{Graph}}   

\DeclareMathOperator{\Diff}{Diff}               
\DeclareMathOperator{\Lie}{\mathcal{L}}         
\DeclareMathOperator{\id}{id}    
\renewcommand{\d}{\mathrm d}                    
\newcommand{\Ss}{\mathbb S}											
\newcommand{\X}{\ensuremath{\mathfrak{X}}}			

\DeclareMathOperator{\Ver}{Ver}        

\newcommand{\G}{\mathcal{G}}            
\newcommand{\M}{\mathcal{M}}            
\newcommand{\Mt}{\widetilde{\mathcal{M}}}            
\newcommand{\s}{\mathbf{s}}             
\renewcommand{\t}{\mathbf{t}}           
\renewcommand{\F}{\mathcal{F}}          
\newcommand{\al}{\alpha}                
\newcommand{\be}{\beta}                 

\newcommand{\T}{\mathbb{T}}             
\newcommand{\TM}{\mathbb{T}M}             

\renewcommand{\gg}{\mathfrak{g}}        
\newcommand{\g}{\mathfrak{g}}           

\DeclareMathOperator{\Ker}{Ker}         
\DeclareMathOperator{\im}{Im}           

\DeclareMathOperator{\var}{var}		

\newcommand{\ri}{\overrightarrow}
\renewcommand{\le}{\overleftarrow}

\numberwithin{equation}{section}

\begin{document}
\DeclareGraphicsExtensions{.pdf,.png}

\title[Integrability and Reduction of Hamiltonian Actions]{Integrability and Reduction of Hamiltonian Actions on Dirac Manifolds}
\author{Rui Loja Fernandes}
\address{Department of Mathematics,
University of Illinois at Urbana-Champaign,
1409 W.~Green Street,
Urbana, IL 61801
USA } 
\email{ruiloja@illinois.edu}
\author{Olivier Brahic}
\address{Department of Mathematics, Federal University of Paran\'a.
CP 19081, 81531-980, Curitiba, PR, Brazil.}
\email{brahicolivier@gmail}
\thanks{Partially supported by NSF grant DMS 1308472, and by CNPq 
grant 401253/2012-0.}

\begin{abstract}
For a Hamiltonian, proper and free action of a Lie group $G$ on a Dirac manifold $(M,L)$, with a regular moment map $\mu:M\to \gg^*$, the manifolds $M/G$, $\mu^{-1}(0)$ and $\mu^{-1}(0)/G$ all have natural induced Dirac structures. If $(M,L)$ is an integrable Dirac structure, we show that $M/G$ is always integrable, but $\mu^{-1}(0)$ and $\mu^{-1}(0)/G$ may fail to be integrable, and we describe the obstructions to their integrability.
\end{abstract}

\maketitle

\section{Introduction}             %
\label{sec:introduction}           %

Let $(M,L)$ be a Dirac manifold. If a Lie group $G$ acts on $M$ by Dirac automorphisms, in a proper and free fashion, one may hope that $M/G$ inherits a Dirac structure from $M$. However, this is not true as can be seen by simple examples of Dirac structures associated with closed 2-forms. One needs to impose extra conditions. 

Consider the map $j:L\to\gg^*$ which to a pair $(X,\al)\in L$ associates $\xi\mapsto i_{\xi_M}\al$, where $\xi_M$ denotes the infinitesimal generator associated with $\xi\in\gg$. This map is a Lie algebroid morphism whenever the action is by Dirac automorphisms. We call a Dirac action \textbf{regular} if $j$ is fiberwise surjective, in which case it follows that there is an induced Dirac structure on the quotient $M/G$ such that $L_{M/G}=\pi_*L$. It is then natural to ask if the quotient $(M/G,L_{M/G})$ is integrable whenever $(M,L)$ is integrable. This is indeed the case since one has the following result: 

\begin{thm}[\cite{BC}]
\label{thm:integration:Dirac}
Let $(M,L)$ be an integrable Dirac manifold and assume that $G$ acts on $M$ by Dirac automorphisms. If the action is regular, proper and free, then $(M/G,L_{M/G})$ is an integrable Dirac structure.
\end{thm}

This result, on the one hand, generalizes a result in \cite{FeOrRa} for the case of Poisson manifolds. On the other hand, it is a special case of a result proved in \cite{BC} for more general actions. Here we consider only the case of regular actions, since they suffice for our main object of study: Hamiltonian actions. 

An integration of the quotient Dirac manifold $(M/G,L_{M/G})$ can be described as follows: if $(\G(L),\Omega)\tto (M,L)$ is the Weinstein groupoid of $(M,L)$, the $G$-action on $(M,L)$ lifts to a proper and free Hamiltonian $G$-action on $(\G(L),\Omega)$, by groupoid automorphisms, with moment map the unique groupoid morphism $J:\G(L)\to\gg^*$ such that $J_*=j$. The regularity assumption guarantees that $0$ is indeed a regular value of $J$ so that the Hamiltonian quotient $(J^{-1}(0)/G,\Omega_{\text{red}})$ is a presymplectic groupoid integrating $M/G$. Briefly, $\G(L)\red G$ integrates $M/G$.

Let us now turn to our main object of study. By a \textbf{Hamiltonian action} we mean a smooth action of a connected Lie group $G$ on the Dirac manifold $(M,L)$ for which there is a $G$-equivariant map $\mu:M\to\gg^*$ satisfying the \textbf{moment map} condition:
\[
(\xi_M,\d\mu_\xi)\in\Gamma(L),\quad \forall \xi\in\gg.
\]
Here $\mu_\xi:M\to\Rr$ denotes the function $\mu_\xi(\cdot):=\langle \mu(\cdot),\xi\rangle$. We call the quadruple $(M,L,G,\mu)$ a \textbf{Hamiltonian $G$-space}. 

It is easy to give simple examples of proper and free Hamiltonian actions for which the Dirac structure on $M$ does not descend to the quotient $M/G$. So one needs some regularity condition on the moment map for this to happen.  Recall that we can think of $(M,L)$ as a (possible singular) foliation of $M$ by presymplectic leaves. It is easy to see that a Hamiltonian action is always tangent to the leaves and on each leaf $(S,\omega_S)$ the usual moment map condition holds:
\[ i_{\xi_M}\omega_S=\d \mu_\xi|_S. \]
We shall say that the \textbf{moment map is regular at $c\in\gg^*$} if $c$ is a regular value of the restriction $\mu|_S:S\to\gg^*$, for every presymplectic leaf $(S,\omega_S)$. Note that this implies that $c$ is a regular value of $\mu$, in the usual sense, but the converse, in general, does not hold. We will say that the \textbf{moment map is regular} if it is regular at all $c\in\gg^*$. 

It turns out that the moment map of a Hamiltonian action is regular if and only if the Dirac action is regular, and then one has Hamiltonian reduction of Dirac structures in the following form:

\begin{thm}
\label{thm:Dirac:ham:reduction}
Let $(M,L,G,\mu)$ be a Hamiltonian $G$-space and assume that the moment map is regular. If the action is proper and free, then there are unique Dirac structures on $M/G$, $\mu^{-1}(c)$ and $M\red G:=\mu^{-1}(c)/G$ such that for the maps
\[\SelectTips{cm}{}\xymatrix@R=10pt@C=10pt{
                                      & M\ar[dr]^\pi & \\
\mu^{-1}(c)\ar[ur]^{i_c}\ar[dr]_{\pi_c} &     &M/G\\
& M\red G  \ar[ur]_i }\]
one has $L_{\mu^{-1}(c)}=i_c^*L=\pi_c^*L_{M\red G}$, and $L_{M/G}=\pi_*L=i_*L_{M\red G}$.
\end{thm}

Particular cases of this result can be found in the work of Blankenstein and van der Schaft \cite{BS}, with assumptions that seem to be superfluous. This result should also follow from the general results on reduction of Courant algebroids and Dirac structures described by Bursztyn, Cavalcanti and Gualtieri \cite{BCG} and the results on the so called optimal moment map described in Jotz and Ratiu \cite{JR}.  Since we could not find this exact statement in the literature, we have included a direct proof of it in Section \ref{sec:Dirac} below.

Assume now that $(M,L,G,\mu)$ is a Hamiltonian $G$-space and that $(M,L)$ is an integrable Dirac structure, so that its Weinstein groupoid $(\G(L),\Omega)$ is a presymplectic groupoid integrating $(M,L)$ (see \cite{BCWZ}). We will se that we have a lifted $G$ action on $(\G(L),\Omega)$ which is Hamiltonian with moment map $J:\G(L)\to\gg^*$ given by:
\[  J(x)=\mu\circ\t(x)-\mu\circ\s(x), \]
where $\s$ and $\t$ are the source and target maps of $\G(L)\tto M$. Note that the moment map $J$ is still the integration of the Lie algebroid morphism $j:L\to\gg^*$. Hence the regularity condition on the moment map guarantees that $0$ is a regular value of $J$ and we have the following corollary of Theorem \ref{thm:integration:Dirac}:

\begin{cor}
\label{cor:integ:Dirac:quotient}
Let $(M,L,G,\mu)$ be a Hamiltonian $G$-space and assume that $(M,L)$ is integrable. If the action is proper and free and the moment map is regular then the quotient $M/G$ is integrable and the Hamiltonian quotient $(J^{-1}(0)/G,\Omega_\text{red})$ is a presymplectic groupoid integrating $(M/G,L_{M/G})$.
\end{cor}

Given a Hamiltonian $G$-space $(M,L,G,\mu)$, Corollary \ref{cor:integ:Dirac:quotient} shows that the quotient $(M/G,L_{M/G})$ is integrable provided $(M,L)$ is integrable and the action is proper, free and regular. We will give an example below which shows that, in general, the Hamiltonian quotient $(M\red G,L_{M\red G})$ is not an integrable Dirac structure, even if $(M,L)$ is integrable. 

However, the obstructions to the integrability of $M\red G$ can be determined explicitly. We introduce for each $m\in M$ a certain group morphism, called the \textbf{$G$-monodromy morphism}: 
\[ \pi_1(G)\to \G(L)_m, \]
whose image $\Mt_m(G)$ we call the \textbf{$G$-monodromy group} at $m$ of the action. These groups control the integrability of $M\red G$. More precisely, set 
\[ \Mt:=\bigcup_{m\in M}\Mt_m(G)\subset\G(L),\quad \Mt_0:=\Mt|_{\mu^{-1}(0)}. \]
Then we have:

\begin{thm}
\label{thm:integ:Ham:quotient}
Let $(M,L,G,\mu)$ be a Hamiltonian $G$-space and assume that $(M,L)$ is integrable. If the action is proper and free and the moment map is regular, then $(M\red G,L_{M\red G})$ is an integrable Dirac structure if and only if $\Mt_0\subset\G(L)$ is an embedded Lie subgroupoid.
\end{thm}

The obstructions clearly vanish if $G$ has a finite fundamental group. Other conditions for vanishing of the obstructions can be obtain from the description of the $G$-monodromy morphism. We provide a nice geometric description of this morphism when a presymplectic leaf is regular, in terms of variations of the presymplectic areas of disks whose boundary is an orbit of a closed loop in $G$.

When the obstructions vanish we can describe one presymplectic groupoid integrating $(M\red G,L_{M\red G})$ as follows: upon restricting to an appropriate neighborhood of $\mu^{-1}(0)$, we can assume that $\Mt\subset\G(L)$ is an embedded Lie subgroupoid. Then the quotient $\G=\G(L)/\Mt$ is a presymplectic Lie groupoid integrating $(M,L)$ and there is a $G\times G$ action on $\G$ which is Hamiltonian with moment map $\bar{J}:\G\to \gg^*\oplus\gg^*$, given by:
\[ \bar{J}(x)=(\mu\circ\t(x),-\mu\circ\s(x)). \]
This action is proper and free, and $0$ is a regular value of $\bar{J}$. It follows that the Hamiltonian quotient $\G\red G\times G$ is well defined and it is in fact a presymplectic groupoid integrating $M\red G$.

Our results can be summarized in the following table:
\begin{table}[htdp]
\begin{center}
\hskip-.83in
\renewcommand{\arraystretch}{1.5}
\scalebox{.95}{
\begin{tabular}{|p{2.0 cm}|p{3.3cm}|l|}
\hline
\emph{Data}                                 		& \emph{Assumptions}  						& \emph{Integration} \\
\hline\hline
$G$-space $(M,L,G)$				& $(M,L)$ is integrable \& action is regular			& $\G(L)\red G$ integrates $M/G$\\
\hline
Hamiltonian $G$-space $(M,L,G,\mu)$	& $(M,L)$ is integrable \& moment map is regular	& $\G(L)\red G$ integrates $M/G$\\
\cline{2-3}
								& +  vanishing of obstructions 				& $\G(L)\red (G\times G)$ integrates $M\red G$\\
\hline
\end{tabular}}
\hskip-.75in
\label{results}
\end{center}
\caption{Summary of Results}
\end{table}

This paper is organized as follows. In Section \ref{sec:Dirac}, we review the basics of Dirac structures and their symmetry reduction, proving Theorem \ref{thm:Dirac:ham:reduction}. In Section \ref{sec:Dirac:quotients}, we study the integrability of Dirac quotients $M/G$, and we give a direct proof of Theorem \ref{thm:integration:Dirac}. Finally, in Section \ref{sec:Ham:quotients} we introduce the $G$-monodromy morphism and we show that they give the obstructions to the integrability of Hamiltonian quotients, proving Theorem \ref{thm:integ:Ham:quotient}. We also give an example of Hamiltonian quotient which is non-integrable.
\vskip 10 pt

{\bf Acknowledgements.} We would like to thank Henrique Bursztyn and David Iglesias-Ponte for several comments and remarks on a preliminary version of this paper.


\section{Dirac Structures and Reduction}  
\label{sec:Dirac}


In this section we will state and give short proofs of some results concerning reduction of Dirac structures, that we use in this paper and for which we could not find references in the literature. Particular cases of the results stated here can be found in the work of Blankenstein and van der Schaft \cite{BS}, with assumptions that seem to be superfluous. The results stated here should also follow from the general results on reduction of Courant algebroids and Dirac structures described by Bursztyn, Cavalcanti and Gualtieri \cite{BCG} and the results on the so called optimal moment map described in Jotz and Ratiu \cite{JR}.

\subsection{Dirac Structures} Let us start by recalling the definition of Dirac structures, mainly to fix notations and sign conventions. Recall that Dirac structures were introduced by T.~Courant in \cite{Cou} as geometric structures generalizing both Poisson structures and presymplectic structures. Dirac structures on a manifold $M$ are described by certain subundles of the generalized tangent bundle $\TM:=TM\oplus T^*\!M$. This bundle comes equipped with natural pairings $\langle~,~\rangle_\pm$ on
its space of sections $\X(M)\times\Omega^1(M)$, defined by:
\begin{equation}
 \label{eq:pairings}
\langle (X,\al),(Y,\be)\rangle_\pm:=\frac{1}{2}\left(i_Y\al\pm i_X\be\right).
\end{equation}
The space of sections of $\TM$ also carries a skew-symmetric bracket called the \textbf{Courant bracket} \cite{Cou}, which is given by:
\begin{equation}
  \label{eq:Courant:bracket}
  \llbracket(X,\al),(Y,\be)\rrbracket:=
  \bigl([X,Y],\Lie_X\be-\Lie_Y\al+\d\langle (X,\al),(Y,\be)\rangle_-\bigr)
\end{equation}
In general, this bracket \emph{does not} satisfy the Jacobi identity. 

\begin{defn}
Let $L\subset \TM$ be a subundle. We say that:
\begin{enumerate}[(i)]
\item $L$ is an \textbf{almost Dirac structure} on $M$ if it is maximal isotropic with respect to $\langle~,~\rangle_+$,
\item $L$ is a \textbf{Dirac structure} on $M$ if it is an almost Dirac structure and $\Gamma(L)$ is closed under the bracket $\llbracket~,~\rrbracket$.
\end{enumerate}
\end{defn}

For a Dirac structure $L$, the restriction of the bracket $\llbracket~,~\rrbracket$ to $\Gamma(L)$ yields a Lie bracket and if we let $\sharp:L\to TM$ be the restriction of
the projection to $TM$, then $(L,\llbracket~,~\rrbracket,\sharp)$ defines a Lie algebroid. Each leaf of the corresponding characteristic foliation (obtained by integrating the
singular distribution $\im\sharp$) carries a pre-symplectic form $\omega$: if $X,Y\in \im\sharp$, we can choose $\al,\be\in T^*\!M$ such that
$(X,\al),(Y,\be)\in L$ and set:
\begin{equation}
  \label{eq:pre:symp:form}
  \omega(X,Y):=\langle (X,\al),(Y,\be)\rangle_-=i_Y\al=-i_X\be.
\end{equation}
One can check that this definition is independent of choices and that $\omega$ is indeed closed. Thus we may think of a Dirac manifold as a (singular) foliated
manifold by pre-symplectic leaves.

For Dirac structures, there are two types of maps: forward Dirac maps and backward Dirac maps (see \cite{BuRa}). We recall here the definitions:

\begin{defn}If $\phi:(M,L_M)\to(N,L_N)$ is a smooth map between Dirac manifolds. Then:
\begin{enumerate}[(i)]
\item $\phi$ is called a \textbf{forward Dirac map} or simply a \textbf{f-Dirac map}, if:
\begin{multline*}
  L_N=\phi_*L_M:=\{(Y,\be)\in TN\oplus T^*N:\exists X\in TM\text{ with }\\
  \d\phi\cdot X=Y\text{ and }(X,\d\phi^*\be)\in L_M\}.
\end{multline*}
\item $\phi$ is called a \textbf{backward Dirac map} or simply a \textbf{b-Dirac map}, if:
\begin{multline*}
  L_M=\phi^*L_N:=\{(X,\al)\in TM\oplus T^*\!M:\exists \be\in T^*N\text{ with }\\
  (\d\phi)^*\be=\al\text{ and }(\d\phi\cdot X,\be)\in L_N\}.
\end{multline*}
\end{enumerate}
\end{defn}

For a general map $\phi:M\to N$ these two conditions are distinct. However, for a diffeomorphism these two conditions are equivalent and when they hold we call $\phi$ a \textbf{Dirac diffeomorphism}. 

Let us denote by $\Diff(M,L)$ the group of Dirac automorphisms of a Dirac manifold $(M,L)$. We will say that $Z\in\X(M)$ is a \textbf{Dirac vector field} if its flow $\phi_t^Z$ is a 1-parameter group of Dirac diffeomorphisms. Its is easy to see that this happens iff
\[ (X,\al)\in\Gamma(L)\ \Longrightarrow\ 
\Lie_Z (X,\al):=(\Lie_Z X,\Lie_Z\al)\in\Gamma(L).\]
It is clear the space of Dirac vector fields is closed under
the usual Lie bracket of vector fields, and hence form a Lie
subalgebra $\X(M,L)\subset \X(M)$. We should think of $\X(M,L)$ as
the Lie algebra of the (infinite dimensional) Lie group $\Diff(M,L)$. 

The symmetry group of Dirac structures include, besides the Dirac diffeomorphisms, the \textbf{gauge transformations} or \textbf{$B$-transforms}: given a Dirac structure $(M,L)$ and a closed 2-form $B\in\Omega^2(M)$ one defines the $B$-transform of $L$ to be the Dirac structure:
\[ e^BL:=\{(X,\al+i_XB): (X,\al)\in L\} .\]
It follows that the semi-direct product $\Diff(M)\ltimes\Omega^2_{\mathrm{cl}}(M)$ acts on the space of Dirac structures on M. Here, two pairs $(\phi_1,B_1),(\phi_2,B_2)\in\Diff(M)\ltimes\Omega^2_{\mathrm{cl}}(M)$ compose as:
\[ (\phi_1,B_1)\cdot(\phi_2,B_2):=(\phi_1\circ\phi_2, \phi_2^*B_1+B_2) \]

Let us recall some of the main examples of Dirac structures.

\begin{ex}
\label{ex:Dirac:2-forms}
A 2-form $\omega\in\Omega^2(M)$ defines an almost Dirac structure:
\[ 
L_\omega:=\graph(\omega)=\set{(X,\al)\in \TM:\al=i_X\omega}.
\] 
$L_\omega$ is a Dirac structure if and only if $\omega$ is closed. This Dirac structure is characterized by the fact that it has only one presymplectic leaf. The effect of a $B$-transform on $L_\omega$ is to change it to $L_{\omega+B}$. 
\end{ex}

\begin{ex}
\label{ex:Dirac:Poisson}
A bi-vector $\pi\in\X^2(M)$ defines an almost Dirac structure:
\[ 
L_\pi:=\graph(\pi)=\set{(X,\al)\in \TM:X=\pi(\al,\cdot)}.
\] 
$L_\pi$ is a Dirac structure if and only if $[\pi,\pi]=0$, i.e., $\pi$ is a Poisson structure. In this case, the presymplectic leaves are in fact symplectic, and we recover the symplectic foliation of a Poisson manifold. A $B$-transform on $L_\pi$ produces a new Dirac structure with the same underlying foliation, but where the 2-form $\omega_S$ on a leaf $S$ is changed to $\omega_S+B|_S$ (which may or may not be symplectic, so $e^BL_\pi$ mail fail to be Poisson).
\end{ex}

\begin{ex}
\label{ex:Dirac:foliations}
Let $D\subset TM$ be a distribution. Then $D$ together with its annihilator defines an almost Dirac structure:
\[ L_D=D\oplus D^0=\set{(X,\al)\in \TM:X\in D, \al\in D^0}. \]
$L_D$ is a Dirac structure if and only if $D$ is an integrable distribution: $D=T\F$, where $\F$ is a regular foliation. In this case, $L_{\F}=T\F\oplus\nu^*(\F)$ is a Dirac structure whose presymplectic foliation is $\F$ itself, where the leaves have zero presymplectic forms. A $B$-transform has the effect of adding to the leaves a pre-symplectic form, namely the restriction of $B$ to each leaf.
\end{ex}

\begin{ex}
We can also combine all the constructions above as follows: let $D\subset TM$ be a distribution, let  $\omega\in\Omega^2(M)$ be a 2-form and let  $\pi\in\X^2(M)$ be a bi-vector field. If $\pi^\sharp\circ\omega^\flat(D)\subset D$, then:
\[ L:=\set{(X+\pi^{\sharp}(\al),i_X\omega+\al)\in \TM:X\in D, \al\in D^0}, \]
defines an almost Dirac structure. If $D$ is integrable, $\omega$ is closed and $\pi$ is Poisson, then $L$ is a Dirac structure. 
\end{ex}

\subsection{Symmetries of Dirac Structures}
Let $G$ be a Lie group acting on a Dirac manifold $(M,L)$. We denote by $\xi_M\in\X(M)$ the infinitesimal generator associated with an element in the Lie algebra $\xi\in\gg$. When $G$ acts by Dirac diffeomorphisms on $(M,L)$ we will call the action a \textbf{Dirac action}. From the discussion in the previous paragraph it is clear what is the infinitesimal version of a Dirac action:

\begin{prop}
For a Dirac action of $G$ on $(M,L)$:
\begin{equation}
\label{infin:generator}
(X,\al)\in L\ \Longrightarrow\ (\Lie_{\xi_M}X,\Lie_{\xi_M}\al)\in L, \quad \forall \xi\in\gg.
\end{equation}
Conversely, if $G$ is a connected Lie group and this condition holds then $G$ acts by Dirac diffeomorphisms on $(M,L)$.
\end{prop}

There is another useful way of translating the infinitesimal condition \eqref{infin:generator} for $G$ to act by Dirac automorphisms. For that, given any action of $G$ on $(M,L)$ we will denote by $j:L\to\gg^*$ the map:
\[ (X,\al)\stackrel{j}{\longmapsto} (\xi\mapsto i_{\xi_M}\al). \]
Then:

\begin{prop}
Let $G$ act on $(M,L)$. Then \eqref{infin:generator} holds if and only if $j:L\to\gg^*$ is a Lie algebroid morphism (where $\gg^*$ is considered as an abelian Lie algebra).
\end{prop}

\begin{proof}
The map $j$ is a Lie algebroid morphism if and only if it commutes with the Lie algebroid differentials:
\[ \d j^*=j^*\d .\]
and it is enough to check this condition in degrees $0$ and $1$. 

Since $\gg^*$ is a Lie algebroid over $\{*\}$ this condition is trivial in degree $0$. On the other hand, in degree $1$, given $\xi\in\gg$ we view it as 1-form $\xi: \gg^*\to\Rr$, and a straightforward computation gives:
\begin{multline*} 
\d(j^*\xi)((X,\al),(Y,\be))=(j^*\d\xi)((X,\al),(Y,\al)) \\
\Longleftrightarrow\quad \langle (\Lie_{\xi_M}X,\Lie_{\xi_M}\al),(Y,\be)\rangle_+=0
\end{multline*}
for all $(X,\al),(Y,\be)\in\Gamma(L)$. So $j$ is a Lie algebroid morphism if and only if \eqref{infin:generator} holds.
\end{proof}

The morphism $j:L\to\gg^*$ is a kind of moment map. In fact, in the case where $L$ is the Dirac structure associated with a Poisson structure, we have
$L=T^*M$, and $j$ is precisely the moment map for the cotangent lift of the action. Another facet of this, which we will discuss later, happens when $(M,L)$ is integrable, in which case $j$ is the infinitesimal version of a moment map $J:\G\to\gg^*$ defined on the presymplectic groupoid integrating $(M,L)$.

\subsection{Regular Dirac actions}
Given a Dirac action of $G$ on $(M,L)$ which is proper and free, so that $M/G$ is a smooth manifold, we may ask if there exists a Dirac structure $L_{M/G}$ such that the quotient map $\pi:M\to M/G$ if an $f$-Dirac map. At each point $[p]\in M/G$ the fiber of $L_{M/G}$ is completely determined by the condition that $q$ is f-Dirac:
\begin{equation}
\label{red:Dirac}
L_{M/G}|_{[p]}:=\pi_*L|_{[p]}
=\{(\d_p\pi\cdot X,\al)\in \T_{[p]}M/G: (X,(\d_p\pi)^*\al)\in L|_p\}.
\end{equation}
Each of these vector spaces are maximal isotropic in $\T M/G$. If they form a smooth vector bundle then we obtain an almost Dirac structure which is automatically integrable. So the only issue about reduction of Dirac structures under symmetry is whether \eqref{red:Dirac} is a smooth bundle. 

\begin{ex}
\label{ex:discrete}
If $G$ is a discrete Lie group that acts properly and freely on $(M,L)$ then the Dirac structure is always reducible: $\pi: M\to M/G$ is a local diffeomorphism so it is clear that \eqref{red:Dirac} is a smooth bundle. 
\end{ex}

\begin{ex}
If $\pi$ is a Poisson structure in $M$ which is invariant under a proper and free action of a Lie group $G$, then there is an induced Poisson structure $\pi_{M/G}$ in $M/G$. This corresponds in Dirac language to the fact that the Dirac structure $L_{\pi}$ is $G$-invariant, it is always reducible, and the quotient is still given by a Poisson structure.
\end{ex}

Although Dirac structures associated with Poisson structures are always reducible, this is not the case with our types of Dirac structures as it is illustrated by the following simple example.

\begin{ex}
Let $M=\Rr^2$ and let $L_D=D\oplus D^0$ be the Dirac structure associated with the 1-dimensional distribution $D=\langle x\frac{\partial}{\partial {x}}+\frac{\partial}{\partial {y}}\rangle$. The action of $G=\Rr$ on $\Rr^2$ by translations in the $y$-direction is Dirac. The projection \eqref{red:Dirac} of $L_D$ along $\pi:\Rr^2\to \Rr^2/G=\Rr$, $(x,y)\mapsto x$, is:
\[
(\pi_*L_D)|_x=\left\{\begin{array}{r}
\{(0,a\,\d x)\in \T\Rr: a\in\Rr\},\quad\text{ if }x=0,\\
\\
\{(a\,\frac{\partial}{\partial {x}},0)\in \T\Rr: a\in\Rr\},\quad\text{ if }x\ne0,\\
\end{array}\right.
\]
which is clearly not a smooth subundle.
\end{ex} 

In the previous example, the action moves a leaf to a different leaf. It is easy to see that for Dirac structures associated with an integrable distribution $L_D$, the Dirac structure is reducible whenever the action is tangent the leaves. Our next example shows that for Dirac structures associated with closed 2-forms $L_\omega$, the Dirac structure may fail to be reducible, even if the action is tangent to the leaves.

\begin{ex}
\label{ex:hamiltonian:singular}
Let $M=\Rr^2$ with the Dirac structure $L_\omega$ associated with the closed 2-form $\omega=x\d x\wedge \d y$. Again, the action of $G=\Rr$ on $\Rr^2$ by translations in the $y$-direction is Dirac. The projection \eqref{red:Dirac} of $L_\omega$ along $\pi:\Rr^2\to \Rr^2/G=\Rr$, $(x,y)\mapsto x$, is:
\[
(\pi_*L_\omega)|_x=\left\{\begin{array}{r}
\{(a\,\frac{\partial}{\partial {x}},0)\in \T\Rr: a\in\Rr\},\quad\text{ if }x=0,\\
\\
\{(0,a\,\d x)\in \T\Rr: a\in\Rr\},\quad\text{ if }x\ne0,\\
\end{array}\right.
\]
which is clearly not a smooth subundle.
\end{ex}

Next we study a natural condition that guarantees that a Dirac structure is reducible.

\begin{defn}
\label{defn:regular}
A Dirac action of $G$ on $(M,L)$ is called a \textbf{regular Dirac action} if $j:L\to\gg^*$ is fiberwise surjective.
\end{defn}

This condition is reminiscent of the usual regularity condition for a moment map, and we will return to this idea later when we study Hamiltonian actions and integrability. In fact, it is easy to see that a regular Dirac action is always locally free. When a regular action is proper and free we obtain:

\begin{prop}
\label{prop:Dirac:reducible}
Let $G$ act on $(M,L)$ and assume that the action is a regular Dirac action which is proper and free. Then the Dirac structure is reducible.
\end{prop}

\begin{proof}
It follows from Example \ref{ex:discrete} that we can assume that $G$ is a connected Lie group.

In order to prove that $ \pi_*L$, given by \eqref{red:Dirac}, is a smooth bundle we observe that the regularity assumption implies that $\Ker(j)\subset L$ is a subundle. Now observe that the action of $G$ on $\TM$ leaves this subundle invariant, and it follows from \eqref{moment:regular}, that we have a $G$-invariant surjective bundle map:
\[ \SelectTips{cm}{}\xymatrix{
\Ker(j)\ar[r]\ar[d] & \pi_*L\ar[d]\\
M\ar[r]& M/G} \] 
By passing to quotient, we obtain an isomorphisms $\Ker(\mu)_*/G\simeq \pi_*L$, therefore showing that $\pi_*L$ is a smooth bundle. This proves that there exist a unique Dirac structure $L_{M/G}$ such that the quotient map $M\to M/G$ is f-Dirac.

\end{proof}

The regularity condition is sufficient, but not necessary, to guarantee that $L$ is reducible: for any proper and free action on a manifold equipped with the zero 2-form $\omega\equiv0$, the Dirac structure $L_\omega$ is reducible, while the map $j$ is identically zero. However, the regularity condition is sufficient for us and, in particular, to deal with Hamiltonian actions to which we turn next.

\subsection{Hamiltonian actions}
We are interested in Hamiltonian actions in the following sense:

\begin{defn}
We call a Dirac action of $G$ on $(M,L)$ a \textbf{Hamiltonian action} if there exist a $G$-equivariant map $\mu:M\to\gg^*$ such that the following condition holds for all $\xi\in\gg$:
\begin{equation}
\label{eq:moment:map}
(\xi_M,\d\mu_\xi)\in\Gamma(L)
\end{equation}
where $\mu_\xi:M\to\Rr$ denotes the function $m\mapsto\langle \mu(m),\xi\rangle$. The quadruple $(M,L,G,\mu)$ is called a \textbf{Hamiltonian $G$-space}. 
\end{defn}

Note that the moment map condition \eqref{eq:moment:map} implies that \eqref{infin:generator} holds, and so the action of the connected component of the identity $G^0$ is by Dirac automorphisms. We have added the assumption that $G$ acts by Dirac automorphisms since the group $G$ may fail to be connected. The following result is a straightforward consequence of the $G$-equivariance of the moment map:

\begin{prop}
For a Hamiltonian $G$-space $(M,L,G,\mu)$ the moment map $\mu:M\to\gg^*$ is a f-Dirac map, where on $\gg^*$ we consider the Dirac structure associated with the canonical linear Poisson structure.
\end{prop}

Proper and free hamiltonian actions may fail to be reducible. 

\begin{ex}
The action in Example \ref{ex:hamiltonian:singular} is hamiltonian with moment map $\mu:\Rr^2\to \Rr$ given by $(x,y)\mapsto x^2/2$, but is not reducible.
\end{ex}

The problem with the previous example is that 0 is not a regular value of the moment map. For a general Hamiltonian action we need a slightly stronger condition:

\begin{defn}
Given a Hamitonian $G$-space $(M,L,G,\mu)$, we say that the \textbf{moment map is regular at $c\in\gg^*$} if $c$ is a regular value of the restriction of $\mu|_S$ to each presymplectic leaf. We say that the \textbf{moment map is regular} if it is regular at every $c\in\gg^*$.
\end{defn}

Note that for any Hamitonian $G$-space $(M,L,G,\mu)$, the moment map condition \eqref{eq:moment:map} gives:
\begin{equation}
\label{moment:regular} 
 \d_m\mu_\xi(X)=i_{\xi_M}\al, \quad \forall \xi\in\gg,
\end{equation}
for any $(X,\al)\in L|_m$. This condition can be restated by saying that we have a commutative diagram:
\[\SelectTips{cm}{}
\xymatrix{ L\ar[r]^j \ar[d]_{\sharp}& \gg^* \\
TM\ar[ru]_{\d \mu}&}
\]
where $\sharp$ is projection on $TM$, i.e., the anchor of the Lie algebroid $L$. It follows that:

\begin{lem}
Given a Hamitonian $G$-space $(M,L,G,\mu)$, the moment map is regular if and only if the Dirac action is regular. Moreover, in this case the action is locally free.
\end{lem}

\begin{ex}
\label{ex:regular:assumpt}
If $L_\omega$ is the Dirac structure associated with a closed 2-form, then there is only one leaf, and so a moment map $\mu:M\to\gg^*$ is regular if and only if $\mu:M\to\gg^*$ is a submersion. 

On the other hand, if $L_\pi$ is the Dirac structure associated with a Poisson structure, a moment map $\mu:M\to\gg^*$ is regular if and only if the action is locally free.
\end{ex}

Proposition \ref{prop:Dirac:reducible} shows that for any Hamitonian $G$-space $(M,L,G,\mu)$, for which the action is proper and free and the moment map is regular, the Dirac structure reduces to $M/G$. The following result shows that we also have induced Dirac structures on the Hamiltonian quotients $\mu^{-1}(c)/G_c$. 

\begin{thm}[Hamiltonian Reduction of Dirac Structures]
\label{thm:Dirac:reduction}
Let $G$ be a hamiltonian action of a Lie group $G$ on a Dirac manifold $(M,L)$ with moment map $\mu:M\to\gg^*$. 
\begin{enumerate}[(i)]
\item Assume that the action is proper and free on $\mu^{-1}(c)$ and that the moment map is regular at $c$. Then there are Dirac structures on $\mu^{-1}(c)$ and $\mu^{-1}(c)/G_c$, where $G_c$ denotes isotropy group of $c\in\gg^*$, such that in the following diagram:
\[\SelectTips{cm}{}\xymatrix@R=10pt@C=10pt{
                                      & M\ar@{-->}[dr]^{\pi}& \\
\mu^{-1}(c)\ar[ur]^i \ar[dr]_{\pi_c} &     &M/G \\
& \mu^{-1}(c)/G_c \ar@{-->}[ur]_{i} }\]
the maps on the left are b-Dirac. 
\item If the action is proper and free on $M$ and the moment map is regular, then there is a Dirac structure on $(M/G,L_{M/G})$ such that in the diagram above the maps on the right are f-Dirac.
\end{enumerate}
\end{thm}

\begin{proof}
We give the proof in the case where the action is is proper and free on $M$ and the moment map is regular. As we pointed out above, by Proposition \ref{prop:Dirac:reducible} we have an induced Dirac structure on $M/G$ such that $L_{M/G}=\pi_*L$.
The description of $ \pi_*L$ given by \eqref{red:Dirac} together with the property \eqref{moment:regular}, show that if $(\d\pi\cdot X,\al)\in \pi_*L$, then $\d\mu(X)=0$. This means that $\mu^{-1}(c)/G_c$ is a union of presymplectic leaves of $(M/G,L_{M/G})$ and it follows that there is a unique Dirac structure on $\mu^{-1}(c)/G_c$ such that the inclusion $i:\mu^{-1}(c)/G_c\hookrightarrow M/G$ is a f-Dirac map.  Finally, we can pullback the Dirac structure on $\mu^{-1}(c)/G_c$ under the submersion $\pi_c:\mu^{-1}(c)\to \mu^{-1}(c)/G_c$, to obtain a Dirac structure on $\mu^{-1}(c)$ such that this projection is b-Dirac. A simple diagram chasing with these Dirac structures shows that the inclusion $i:\mu^{-1}(c)\hookrightarrow M$ is a b-Dirac map.
\end{proof}

The previous result yields as special cases the hamiltonian reduction of presymplectic structures and of Poisson structures under the usual regularity assumptions (see Example \ref{ex:regular:assumpt}). We can also combine these two reductions as follows:

\begin{ex}
Let $(P,\pi)$ be a Poisson structure and $(S,\omega)$ be a presymplectic manifold. Assume that $G$ acts on $(P,\pi)$ and on $(S,\omega)$ in a hamiltonian fashion with moment maps $\mu_P:P\to\gg^*$ and $\mu_S:S\to\gg^*$. If at least one of these actions is proper and free and at least one of the moment maps is regular, we see that the diagonal action of $G$ on $(M:=P\times S,L:= L_\pi\oplus L_\omega)$ is proper and free with regular moment map $\mu:=\mu_P+\mu_S$. It follows that $\mu^{-1}(0)/G$ has a Dirac structure $L_0$, which according to Theorem \ref{thm:Dirac:reduction} is characterized by:
\[ i^*L=\pi^*L_0, \]
where $i:\mu^{-1}(0)\hookrightarrow P\times S$ is the inclusion and $\pi:\mu^{-1}(0)\to \mu^{-1}(0)/G$ is the projection.
\end{ex}


\section{Integrability of Dirac Quotients}  
\label{sec:Dirac:quotients}

In this section we consider a Dirac action of a Lie group $G$ on a Dirac manifold $(M,L)$. We will assume that the action is regular (see Definition \ref{defn:regular}), proper and free, so that we have a quotient Dirac structure $(M/G,L_{M/G})$ such that $\pi_*L=L_{M/G}$. The global objects integrating Dirac structures are presymplectic groupoids, as explained in \cite{BCWZ}.  We are interested in understanding the relationship between the presymplectic groupoids integrating $(M,L)$ and the presymplectic groupoids integrating $(M/G,L_{M/G})$.

\subsection{Groupoids}                                                                                                                             

In the sequel, we will denote by $\G\tto M$ a Lie groupoid, with source
and target maps $\s,\t:\G\to M$, identity section $\iota:M\to \G$, $m\mapsto 1_m$,
and inversion $i:\G\to\G$, $x\mapsto x^{-1}$. Composition of two arrows, denoted 
$x\cdot y$, is defined provided $\s(x)=\t(y)$.

Also, we will denote by $p_A:A\mapsto M$ a Lie algebroid with Lie bracket $[~,~]_A$ 
and anchor $\sharp:A\to TM$. Given a Lie groupoid $\G$, its Lie algebroid has bundle
$A(\G):=\Ker\d_{\iota(M)}\s$ and anchor $\sharp:=\d_{\iota(M)}\t$, so that the sections of 
$A(\G)$ can be identified with the right invariant vector fields on $\G$ (and this 
defines the Lie bracket on sections of $A(\G)$).

Let $A\to M$ be a Lie algebroid. We recall the construction of the Weinstein
groupoid $\G(A)$ (see \cite{CrFe2,CrFe3}). This is a topological groupoid, with source 1-connected fibers,
which morally integrates $A$. Indeed, $A$ is integrable iff $\G(A)$ is smooth and in 
this case $A(\G(A))$ is canonically isomorphic to $A$. As a space, $\G(A)$ is the quotient:
\[ \G(A)=P(A)/\sim, \]
where:
\begin{itemize}
\item $P(A)$ denotes the space of \textbf{$A$-paths}, i.e., the paths $a:I\to A$ such that $\sharp(a(t))=\frac{\d }{\d t}p_A(a(t))$. 
\item $\sim$ denotes the equivalence relation on $P(A)$ given by \textbf{$A$-homotopy}.
\end{itemize}
We denote by $[a]_A$ or $[a]$ the elements of $\G(A)$. We recall that $\G(A)$ becomes a \emph{topological groupoid} with the quotient $C^2$-topology, where one sets 
$\s([a])=p_A\circ a(0)$ and $\t([a])=p_A\circ a(1)$ with multiplication given by concatenation. 

Let us recall how $A$-homotopies are defined since this will be essential later. Suppose we are given $\alpha^\epsilon$(t), a time dependent family of sections of $A$ depending on a parameter $\epsilon\in I:=[0,1]$, and $\beta^0(\epsilon)$ a time dependent section of $A$.  
Then there exists a unique solution $\beta=\beta^t(\epsilon)$ of the following 
\emph{evolution equation}:
\begin{equation}
\label{evolution}
\frac{d\alpha}{d\epsilon}-\frac{d\beta}{dt}=[\alpha,\beta],
\end{equation}
with initial condition $\beta^0(\epsilon)$. In fact, it is easily checked that the following integral formula provides a solution:
\begin{equation}
\label{hintegral}\beta^t(\epsilon):=\int_0^t(\psi^{\alpha^\epsilon}_{t,s})_*(\frac{d}{d\epsilon}\alpha^\epsilon(s))ds+(\psi_{t,0}^{\alpha^\epsilon})_*(\beta^0(\epsilon)).
\end{equation}
Here $\psi^{\alpha^\epsilon}$ denotes the flow of the time dependent linear vector field on $A$, associated to the derivation $[\alpha^\epsilon,-]$ of sections of $A$ (see the appendix in \cite{CrFe2} for more details). We emphasize the use of the indices and parameters in the notation: we think of $\alpha$ as an $\epsilon$-family of $t$-time dependent sections of $A$, while we think of $\beta$ as a $t$-family of $\epsilon$-time dependent sections of $A$ (see why below).

\begin{defn}
A family $a^\epsilon:I\mapsto A,\ \epsilon\in I$ of $A$-paths, over $\gamma^\epsilon:I\mapsto B$ is called an \textbf{$A$-homotopy} if and only the unique solution $\beta$ of (\ref{evolution}) with initial condition $\beta^0(\epsilon)=0$ satisfies: 
\begin{equation}\label{hcondition}
\beta^1(\epsilon)_{\gamma^\epsilon(1)}=0, \forall \epsilon\in I.
\end{equation}
Here, $\alpha^\epsilon$ denotes any family of time-dependant sections of $A$ \emph{extending} $a$, that is, such that $\alpha^\epsilon_{\gamma^\epsilon(t)}(t)=a^\epsilon(t)$. 
\end{defn}

One checks that this definition is independent of the choice of $\alpha$ (see \cite{CrFe2}). Moreover, the homotopy is completely determined by $b^t(\epsilon)=\beta^t(\epsilon)_{\gamma^\epsilon(t)}(t)$, so we will denote an $A$-homotopy by $a(\epsilon,t)\d t + b(\epsilon,t)\d \epsilon$, and we will refer to (\ref{hcondition}) as the {\bf homotopy condition}. Note that, if some $\alpha^0$ is fixed, and we are given $\beta$ with $\beta^0=\beta^1=0$, then equation (\ref{evolution}) can also be considered as an evolution equation for $\alpha$, which turns out to induce an homotopy.

\subsection{Presymplectic groupoids}
\label{sub:symp:grpds}
As a general philosophic principle, when a Lie algebroid is associated with some geometric structure one has some extra geometric structure on the associated Lie groupoid. In the case of the Lie algebroid associated with a Dirac structure, this turns out to be a (multiplicative) presymplectic form, as we briefly recall here. For more details see
\cite{BCWZ} and \cite{PoWa}.

\begin{defn} A  $2$-form $\Omega$ on the space of arrows of a Lie groupoid $\G$ is said to be \textbf{multiplicative} if 
\begin{equation}
\label{eq:multiplicative}
m^*\Omega=\text{pr}_1^*\Omega+\text{pr}_2^*\Omega,
\end{equation}
where $m:\G^{(2)}\to\G$ denotes multiplication of composable arrows and $\text{pr}_i:\G^{(2)}\to\G$ the projection on both factors. A \textbf{presymplectic groupoid} is a Lie groupoid endowed with a multiplicative $2$-form such that:
\begin{equation}
\label{eq:non:degeneracy} 
 \Ker \Omega_x\cap \Ker(\d\s)_x\cap \Ker(\d\t)_x=\{0\},
\end{equation}
for any unit $x\in M$.
\end{defn}

The non-degeneracy condition \eqref{eq:non:degeneracy}  implies that $\dim\G=2\dim M$. Also, this non-degeneracy condition is obviously satisfied when $\Omega$ is symplectic. In this case, the multiplicative condition \eqref{eq:multiplicative} is equivalent to requiring the $\graph(m)\subset \G\times\G\times\overline{\G}$ to be a Lagrangian submanifold, and we say that $(\G,\Omega)$ is a \textbf{symplectic groupoid}. 

Henceforth we will say $L$ is an \textbf{integrable Dirac structure} on $M$ if the Lie algebroid $(L,\llbracket~,~\rrbracket,\sharp)$ is integrable by some Lie groupoid. One of the main results of \cite{BCWZ} establishes a one to one correspondence between $\s$-simply connected presymplectic groupoids and integrable Dirac structures, and can be stated as follows:

\begin{thm}
\label{thm:presymplectic:groupoid}
Let $(\G\tto M,\Omega)$ be a presymplectic groupoid. Then:
\begin{enumerate}[(i)]
\item There is a unique Dirac structure $L$ on $M$ such that $A(\G)$ is isomorphic to $L$ and the target map $\t:(\G,\Omega)\to (M,L)$ is a f-Dirac map. 
\item Conversely, given a Dirac manifold $(M,L)$ such that $L$ is an integrable Lie algebroid, the Weinstein groupoid $\G(L)$  has a naturally induced multiplicative presymplectic form, for which the target map is a f-Dirac map. 
\item If $(\G,\Omega)$ is source simply connected then it is isomorphic, as a presymplectic groupoid, to the Weinstein groupoid $\G(L)$ of the underlying Dirac structure.
\end{enumerate}
\end{thm}

The multiplicative presymplectic form $\Omega$ on $\G(L)$ is related to sections of $L$ by the following formulas, that will be useful later on: for any sections $\eta=(v,\alpha),\xi=(w,\beta)\in\Gamma(L)$, and any $X\in T\G$ one has
\begin{align}
\label{eq:left}
\Omega(\le{\eta},X)&=-\alpha(\s_*X)\\
\label{eq:right}
\Omega(\ri{\xi},X)&=\beta(\t_*X),
\end{align}
where $\le{\eta}$ (respectively $\ri{\xi}$) denotes the left (respectively right) invariant vector field on $\G(L)$ associated to $\eta$ (respectively $\xi$). Also, the source and target fibers turn out to be presymplectically orthogonal: 
\[ \omega(\le{\eta},\ri{\xi})=0. \]

In general, it is a hard problem to describe explicitly the integration of a given Dirac structure. However, this is possible in some simple examples.

\begin{ex}
\label{ex:presymplectic:int}
For the Dirac structure associated with a 2-form $\omega\in\Omega^2(M)$, the pair groupoid $\G=M\times M\tto M$ is a presymplectic groupoid integrating $L_\omega$: one defines the multiplicative presymplectic form on $M\times M$ by:
\[ \Omega:=\t^*\omega-\s^*\omega, \]
where the source and target maps are the two projections on $M$. This groupoid is not source simply connected if $M$ is not simply connected: the Weinstein groupoid is the fundamental groupoid $\Pi_1(M)\tto M$, which also carries a presymplectic form defined in the same way.
\end{ex}

\begin{ex}
In the correspondence given by Theorem \ref{thm:presymplectic:groupoid}, Poisson structures correspond to symplectic groupoids. Note, however, that for a Poisson structure $\pi$, the associated Dirac structure $L_\pi$, in general, is not integrable. When it is integrable, it may be quite complicated to give an explicit integration. This is possible in simple cases: for example, a linear Poisson structure on the dual a Lie algebra $\gg^*$ is always integrable and for any Lie group $G$ with Lie algebra $\gg$, the cotangent bundle $T^*G\tto \gg^*$, with its canonical symplectic structure, is a symplectic groupoid integrating $\gg^*$.
\end{ex}

\begin{ex}
Let $D\subset TM$ be a distribution. The corresponding Dirac structure is always integrable: if $\F$ denotes the corresponding foliation and $\Pi_1(\F)$ is the fundamental groupoid of $\F$ then we have the linear holonomy action of $\Pi_1(\F)$ on the conormal bundle $\nu^*(\F)$. The action groupoid $\G=\Pi_1(\F)\ltimes\nu^*(\F)\tto M$ is a presymplectic groupoid integrating $L_D$ with multiplicative presymplectic form given by:
\[ \Omega:=p^*\omega_\text{can}, \]
where $p:\G\to\nu^*(\F)$ denotes the projection and $\omega_\text{can}$ denotes the pull-back of the canonical symplectic form under the inclusion $\nu^*(\F)\hookrightarrow T^*M$.
\end{ex}

\begin{ex}
If $(\G,\Omega)$ is presymplectic groupoid integrating the Dirac structure $(M,L)$ and $B\in\Omega^2(M)$ is a closed 2-form, then the $B$-transform $(M,e^BL)$ is also an integrable Dirac structure: we take the same Lie groupoid $\G$ with a new presymplectic form:
\[ \Omega_B:=\Omega+\t^*B-\s^*B. \]
Hence, $B$-transforms do not affect the integrability of Dirac structures.
\end{ex}

\subsection{Integrability of Dirac Quotients}                                        
\label{sub:Dirac:quotients}

We can now give a complete answer to the problem of integrating quotients of Dirac structures under regular Dirac actions. Namely, we will prove the following result, which also follows from the general set up described in \cite{BC}:

\begin{thm}
\label{thm:int:Dirac:quotients}
Let $(M,L)$ be an integrable Dirac manifold and assume that $G$ acts on $M$ by Dirac automorphisms. If the action is regular, proper and free, then $(M/G,L_{M/G})$ is an integrable Dirac structure.
\end{thm}

We will prove this result by giving an explicit integration of the quotient Dirac structure $(M/G,L_{M/G})$ and this construction is also itself interesting.

\begin{proof}
First we remark that we can assume that $G$ is a connected Lie group. 

Next, we observe that if $(\G(L),\Omega)$ is the Weinstein groupoid of $(M,L)$ then each element $g\in G$ acts on an element $[a]\in\G(L)$: since the action is by Dirac diffeomorphisms, we have an induced action of $G$ on the bundle $L\to M$, by Lie algebroid automorphisms, that covers the original $G$-action. Hence, if $a:I\to L$ is an $A$-path, then we have the $A$-path $g\cdot a:I\to A$:
\[ (g\cdot a)(t):=g\cdot a(t). \]
Moreover, this action takes $A$-homotopies to $A$-homotopies so that we can set $g\cdot[a]:=[g\cdot a]$. Clearly, this action is by Lie groupoid automorphisms, since the action on $A$-paths intertwines concatenation of $A$-paths. Since the original $G$-action is proper and free, so is the lifted $G$-action on $\G(L)$.

We claim that:
\begin{lem}
\label{lem:hamitlonian}
The $G$-action on $(\G(L),\Omega)$ is a Hamiltonian action with moment map $J:\G(L)\to \gg^*$ the Lie groupoid morphism integrating the Lie algebroid morphism $j:L\to\gg^*$:
\begin{equation}
\label{eq:ham:action}
i_{\xi_\G}\Omega=\d J_\xi, \forall \xi\in\gg. 
\end{equation}
\end{lem}
We defer the proof of this lemma until the end of the proof.

Since $J:\G(L)\to \gg^*$ is a Lie algebroid morphism, its kernel is a subgroupoid of $\G(L)$. We claim that $0$ is a regular value of $J$, so that $J^{-1}(0)$ is a Lie subgroupoid with Lie algebroid $\Ker j\subset L$. In fact, our regularity assumption implies that $j$ is fiberwise surjective, so it follows that $\d_x J$ is surjective for any identity element $x\in\G(L)$. Since $J$ is a groupoid morphism, this implies that $\d_x J$ is surjective for any $x\in\G(L)$, so $J$ is a submersion.

We conclude that the Hamiltonian quotient $(J^{-1}(0)/G,\Omega_\text{red})$ is a presymplectic Lie groupoid over $M/G$ with Lie algebroid $\Ker j/G\simeq L_{M/G}$. This proves that $(M/G,L_{M/G})$ is an integrable Dirac structure.
\end{proof}

\begin{proof}[Proof of Lemma \ref{lem:hamitlonian}]
In order to prove the formula, we observe that for a fixed $\xi\in\gg$ both sides of \eqref{eq:ham:action} are multiplicative 1-forms in $\G$: the right-hand side is the differential of a groupoid homomorphism $\G\to\Rr$ while the left-hand side is the contraction of multiplicative 2-form with an infinitesimal groupoid automorphism. It follows (see \cite{BCWZ}) that it is enough to check that the \eqref{eq:ham:action} holds for tangent vectors $X\in T_{\iota(m)}\G(L)$. Since this tangent space splits into vectors tangent to the identity section and vectors tangent to the source fibers, we consider these two cases separately:
\begin{enumerate}[(a)]
\item $X\in TM$: in this case, we see that the left-hand side of \eqref{eq:ham:action} becomes $\Omega(\xi_\G,X)=0$, since the restriction of $\Omega$ to the identity section vanishes and both $X$ and $\xi_\G$ are tangent to $M$. On the other hand, the right-hand side of \eqref{eq:ham:action} becomes $\d J(X)$ which also vanishes because $J(x)=0$ for any identity arrow $x$.
\item $X\in \Ker\d_{\iota(m)}\s=L_m$: In this case, writing $X=(v,\al)$, the left-hand side of \eqref{eq:ham:action} becomes $\Omega(\xi_\G,X)=i_{\xi_M}\al$, while the right-hand side of \eqref{eq:ham:action} becomes $\d J(X)=j(v,\al)$. The definition of $j$ shows that both sides are equal.
\end{enumerate}
\end{proof}

One can check easily that the reduced presymplectic form $\Omega_\text{red}$ satisfies the non-degeneracy condition \eqref{eq:non:degeneracy}, and that the target map
\[ \t: (J^{-1}(0)/G,\Omega_\text{red})\to (M/G,L_{M/G}) \] 
is a $f$-Dirac map, so that $(J^{-1}(0)/G,\Omega_\text{red})$ is a presymplectic integration of $L_{M/G}$. However, in general, the groupoid $\G(L)\red G:=J^{-1}(0)/G$ is not source simply connected (and, in fact, it may even not be source connected), so that in general:
\[ \G(M,L)\red G \not = \G(M/G,L_{M/G}). \]
One can give conditions for this equality to hold. These are entirely similar to the ones found in \cite{FeOrRa} for the case of Poisson manifolds, so we will not consider this question here.


\section{Integrability of Hamiltonian Quotients}  
\label{sec:Ham:quotients}


We now turn to Hamiltonian actions and to the question of integrability of Hamiltonian quotients. Our aim is to prove Theorem \ref{thm:integ:Ham:quotient} from the Introduction and to explain the obstructions to integrability of Hamiltonian quotients, giving  geometric interpretation in the case of regular presymplectic leaves. Henceforth, we will assume that $(M,L,G,\mu)$ is a Hamiltonian $G$-space, with $G$ a connected Lie group, and that $(M,L)$ is an integrable Dirac structure, so that its Weinstein groupoid $(\G(L),\Omega)$ is a presymplectic groupoid integrating.

\subsection{Lifting of Hamiltonian actions}                                          
\label{sub:lift:ham}                                     

We saw in the previous section that any Dirac $G$-action on $(M,L)$ lifts to a Hamiltonian action of $G$ on $\G(L)$, which is Hamiltonian with moment map $J:\G(L)\to\gg^*$, a Lie groupoid morphism integrating the Lie algebroid morphism $j:L\to\gg^*$. For Hamiltonian $G$-spaces $(M,L,G,\mu)$ we have additionally:

\begin{prop}
For a Hamiltonian $G$-space $(M,L,G,\mu)$ the moment map $J:\G(L)\to\gg^*$ is an exact 1-cocycle:
\[  J(x)=\mu\circ\t(x)-\mu\circ\s(x), \]
\end{prop}

\begin{proof}
Since both $J:\G(L)\to\gg^*$ and $\mu\circ\t-\mu\circ\s:\G(L)\to\gg^*$ are Lie groupoid morphisms, it is enough to check they induce the same Lie algebroid morphism $L\to \gg^*$. If one computes the differential of both sides and restricts to $L=\Ker\d\s$, one obtains that:
\[ j=\d\mu\circ\sharp, \]
which is exactly the moment map condition in the form \eqref{moment:regular}.
\end{proof}

Another form of the moment map condition is the following. The $G$-action on $M$ defines an action groupoid $G\ltimes M\tto M$ with Lie algebroid $\gg\ltimes M\to M$. The map $\mu:M\to\gg^*$ induces a Lie algebroid homomorphism from the action algebroid $\gg\ltimes M\to M$ to the Lie algebroid $L\to M$ by:
\[ \psi_\mu:\gg\ltimes M\to L,\quad (\xi,m)\mapsto (\xi_M|_m,\d_m\mu_{\xi}). \]
In general, this Lie algebroid morphism cannot be integrated to a Lie groupoid morphism $G\ltimes M\to \G(L)$. However, denoting by $\tilde{G}\ltimes M\tto M$ the action groupoid associated with the simply connected Lie group with Lie algebra $\gg$, we do have a Lie groupoid morphism:
\begin{equation}
\label{eq:inner:morphism}
\tilde{\Psi}_\mu: \tilde{G}\ltimes M\to \G(L), \quad (m,[\gamma])\mapsto [a]
\end{equation}
where $a:I\to L$ is the $A$-path:
\[ t\mapsto (\xi(t)_M|_{\gamma(t)\cdot m},\d_{\gamma(t)\cdot m}\mu_{\xi(t)}), \]
where $\xi(t)=\dot{\gamma}(t)\gamma(t)^{-1}\in\gg$. Let us write $q:\tilde{G}\to G$, $[\gamma]\mapsto \gamma(1)$, for the covering map so that its kernel is $q^{-1}(1)=\pi_1(G)\subset C(\tilde{G})$. It follows that:

\begin{prop}\label{prop:ham:inner}
For a Hamiltonian $G$-space $(M,L,G,\mu)$ the lifted $G$-action on $\G(L)$ is inner: for any $g\in G$ and $[a]\in\G(L)$ one has:
\begin{equation}
\label{lift:inner:action}
g\cdot [a]= \tilde{\Psi}_\mu([\gamma],\t([a]))\cdot [a] \cdot \tilde{\Psi}_\mu([\gamma]^{-1},\s([a])).
\end{equation}
if $[\gamma]\in\tilde{G}$ is any element such that $\gamma(1)=g$. In particular, the image $\tilde{\Psi}_\mu(M\times\pi_1(G))\subset \G(L)$ is a normal Lie group bundle, whose fiber at each $m\in M$ is contained in the center of the isotropy group $\G(L)_m$:
\[ \tilde{\Psi}_\mu([\gamma],\t([a]))\cdot [a]=[a]\cdot \tilde{\Psi}_\mu([\gamma],\s([a])),\quad \forall [\gamma]\in\pi_1(G),[a]\in\G(L). \]
\end{prop}

\begin{proof}
Note, on the one hand, that the action of $\tilde{G}$ on $M$ factors through the action of $G$ on $M$. It follows that we have a lifted action of $\tilde{G}$ on $\G(L)$, defined by:
\[ [\gamma]\cdot [a]=\gamma(1)\cdot [a]. \]
In particular, elements of $\pi_1(G)$ act trivially on $\G(L)$. It follows that this $\tilde{G}$-action is by groupoid automorphisms and Hamiltonian with moment map $J:\G(L)\to\gg^*$.

On the other hand, one checks that the right hand side of formula \eqref{lift:inner:action} defines an action of $\tilde{G}$ on $\G(L)$ by groupoid automorphisms which is Hamiltonian with moment map $J:\G(L)\to\gg^*$. 

Since $\tilde{G}$ is connected, these two actions must coincide, and the result follows.
\end{proof}

Let us observe now that we can build a $\tilde{G}\times \tilde{G}$-action on $\G(L)$ by setting:
\[ ([\gamma_1],[\gamma_2])\cdot [a]= \tilde{\Psi}_\mu([\gamma_1],\t([a]))\cdot [a] \cdot \tilde{\Psi}_\mu([\gamma_2]^{-1},\s([a])).\]
This action extends the $G$-action on $\G(L)$, but it is not by groupoid automorphisms. However, it still is a Hamiltonian action:

\begin{prop}
\label{prop:ham:GG}
For a Hamiltonian $G$-space $(M,L,G,\mu)$ the $\tilde{G}\times \tilde{G}$-action on $(\G(L),\Omega)$ is Hamiltonian with moment map $\tilde{J}:\G(L)\to\gg^*\oplus\gg^*$ given by:
\[ \tilde{J}= (\mu\circ\t,-\mu\circ\s).\]
\end{prop}

\begin{proof}
The proof follows by some straight forward computations: first, one checks that the infinitesimal generators of the $\tilde{G}\times \tilde{G}$-action on $\G(L)$ are the vector fields:
\[ (\xi_1,\xi_2)_{\G(L)}=\le{\eta_1}-\ri{\eta_2}, \]
where $\eta_i=((\xi_i)_M,\d\mu_{\xi_i})$. Then using \eqref{eq:left} and \eqref{eq:right} we find that:
\[ i_{(\xi_1,\xi_2)_{\G(L)}}\Omega=\d(\mu_{\xi_1}\circ\t)-\d(\mu_{\xi_2}\circ\s), \]
so the result follows.
\end{proof}

If the action of $G$ on $M$ is free and proper, then the lifted action of $G$ on $\G(L)$ is also free and proper. However, the $\tilde{G}$-action, and also the $\tilde{G}\times \tilde{G}$-action, may fail to be proper and/or free. For this reason, we would like to have a groupoid homomorphism $\Psi_\mu:G\times M\to \G$, where $\G$ is some presymplectic groupoid integrating $(M,L)$, which would allow us to define a $G\times G$-action. 

\begin{prop}
\label{prop:embedded}
For a Hamiltonian $G$-space $(M,L,G,\mu)$, the Lie algebroid morphism $\psi_\mu:\gg\ltimes M\to\G(L)$ integrates to a Lie groupoid morphism $\Psi_\mu:G\ltimes M\to \G$, where $\G$ is some presymplectic groupoid integrating $(M,L)$ if and only if $\Mt:=\tilde{\Psi}_\mu(M\times\pi_1(G))\subset \G(L)$ is an embedded Lie subgroupoid.
\end{prop}

\begin{rem}
The condition that $\Mt:=\tilde{\Psi}_\mu(M\times\pi_1(G))\subset \G(L)$ is an embedded Lie subgroupoid implies that, for each $m\in M$, the subgroup $\Mt_m\subset\G(L)_m$ is discrete. In fact this conditions is equivalent to the existence of an open neighborhood $U$ of the identity section in $\G(L)$ such that $U\cap (\Mt-\iota(M))=\emptyset$. We will see later that this condition is precisely the source of the obstructions to integrability of Hamiltonian quotients.
\end{rem}

\begin{proof}
If $\Mt:=\tilde{\Psi}_\mu(M\times\pi_1(G))\subset \G(L)$ is an embedded Lie subgroupoid then the quotient $\G:=\G(L)/\Mt$ is a Lie groupoid integrating $L$. It follows that the Lie groupoid morphism $\tilde{\Psi}_\mu:\tilde{G}\ltimes M\to \G(L)$ descends to a Lie groupoid morphism $\Psi_\mu:G\ltimes M\to \G$: 
\[
\xymatrix{\SelectTips{cm}{}
\tilde{G}\ltimes M\ar[r]^{\tilde{\Psi}_\mu}\ar[d]& \G(L)\ar[d]\\
G\ltimes M\ar[r]^-{\Psi_\mu}& *+[r]{\G=\G(L)/\Mt}
}
\]
Moreover, the pullback of $\Omega$ to $\Mt\subset\G(L)$ vanishes, so we have an induced presymplectic form on the quotient $\G$, which is still multiplicative and satisfies the non-degeneracy condition. Briefly, $\G$ is a presymplectic groupoid.

For the converse, if there exists a Lie groupoid morphism $\Psi_\mu:G\ltimes M\to \G$, into some presymplectic groupoid $\G$ integrating $(M,L)$, then we will have a commutative diagram as above. It will follow that $\Delta:=\tilde{\Psi}_\mu(M\times\pi_1(G))\subset \G(L)$ is contained in the kernel of the covering map $\G(L)\to \G$ and hence it must be an embedded Lie subgroupoid.
\end{proof}

\subsection{The Integration of Hamiltonian quotients}                        
\label{sub:int:ham}                                     

Let $(M,L,G,\mu)$ be a Hamiltonian $G$-space. Assume that $(M,L)$ is integrable and the action is proper, free and regular. We will use the same notations as in the previous paragraph. In particular, we denote by $\tilde{\Psi}_\mu: \tilde{G}\ltimes M\to \G(L)$ the Lie groupoid morphism integrating the Lie algebroid morphism:
\[ \psi_\mu:\gg\ltimes M\to L,\quad (\xi,m)\mapsto (\xi_M|_m,\d_m\mu_{\xi}). \]

\begin{defn}
The \textbf{$G$-monodromy morphism} at $m$ is the map:
\[ \partial_m:\pi_1(G)\to \G(L)_m, \quad  [\gamma]\mapsto \tilde{\Psi}_\mu([\gamma],m). \]
The image $\Mt_m(G)=\im\partial_m$ is called the \textbf{$G$-monodromy group} at $m$ of the Hamiltonian action. 
\end{defn}

In particular, if $\Mt=\bigcup_{m\in M}\Mt_m$, we have that:
\[ \Mt=\tilde{\Psi}_\mu(M\times\pi_1(G))\subset \G(L). \]
Note that by Proposition \ref{prop:ham:inner}, $\Mt$ is a normal Lie subgroupoid of $\G(L)$ and each $G$-monodromy group $\Mt_m$ is a subgroup of the center of the isotropy group $\G(L)_m$.

To state our main result in this section we also set:
\[  \Mt_0:=\Mt|_{\mu^{-1}(0)}. \]
Then we have:

\begin{thm}
Let $(M,L,G,\mu)$ be a Hamiltonian $G$-space and assume that $(M,L)$ is integrable. If the action is proper and free and the moment map is regular, then $(M\red G,L_{M\red G})$ is an integrable Dirac structure if $\Mt_0\subset\G(L)$ is an embedded Lie subgroupoid.
\end{thm}

\begin{proof}
Assume first that $\Mt_0\subset\G(L)$ is an embedded Lie subgroupoid. We can assume further that there exists a open neighborhood $U$ of $\mu^{-1}(0)$ where $\Mt|_U\subset\G(L)_U$ is an embedded Lie subgroupoid. To simplify the notation we assume that $U=M$. Then, by Proposition \ref{prop:embedded}, we have that $\G:=\G(L)/\Mt$ is a presymplectic groupoid integrating $(M,L)$ and we have a groupoid morphism:
\[ \Psi_\mu:G\ltimes M\to \G, \]
This allows us to define an action of $G\times G$ on $\G$ by setting:
\[ (g_1,g_2)\cdot x= \Psi_\mu(g_1,\t(x))\cdot x \cdot \Psi_\mu(g_2^{-1},\s(x)). \]
Furthermore, by Proposition \ref{prop:ham:GG}, this action is Hamiltonian with moment map $\bar{J}:\G\to \gg^*\oplus\gg^*$, given by:
\[ \bar{J}(x)=(\mu\circ\t(x),-\mu\circ\s(x)). \]
Since the original $G$-action is proper and free, this $G\times G$-action is also proper and free. The regularity assumption on the moment map $\mu$ implies that $0$ is a regular value of $\bar{J}$, so the Hamiltonian quotient $\G\red G\times G:=\bar{J}^{-1}(0)/G$ is a presymplectic manifold. We claim that it is in fact a presymplectic groupoid integrating $M\red G$.

First, notice that if $\bar{J}(x)= \bar{J}(y)=0$ then $\bar{J}(x\cdot y)=0$ and $\bar{J}(x^{-1})=0$, so $\bar{J}^{-1}(0)\subset \G$ is a Lie subgroupoid. Next, observe that if $x,y,x',y'\in\G$ are such that $x\cdot y$ and $x'\cdot y'$ are defined, and if $x'=(g_1,g_2)x$ and $y'=(h_1,h_2)y$, then one must have $g_2=h_1$ and it follows that $x'\cdot y'=(g_1,h_2)x\cdot y$. Hence, the groupoid operation descends to a well defined groupoid operation in the quotient $\bar{J}^{-1}(0)/G$. Some diagram chasing gives that the reduced presymplectic structure $\Omega_\text{red}$ is multiplicative and non-degenerate, and also that the target map $\t:(\G\red G,\Omega_\text{red})\to (M\red G,L_{M\red G})$ is a forward Dirac map. 
\end{proof}

\begin{rem}
The converse of this result is also true: if $(M\red G,L_{M\red G})$ is an integrable Dirac structure then $\Mt_0\subset\G(L)$ is an embedded Lie subgroupoid. The proof is a bit more involved, so we give here a sketch only in the case where $J^{-1}(0)\subset\G(L)$ has source 1-connected fibers. In the general case, there are certain bundles of discrete groups that one must also take into consideration, which are related with the issue of when does the equality $\G(M)\red G=\G(M/G)$ hold (see \cite{FeOrRa} for the case of Poisson manifolds). 

Under this assumption one sees that, on the one hand, $(\mu^{-1}(0),i_0^*L)$ is integrated by $\tilde{J}^{-1}(0)=J^{-1}(0)|_{\mu{-1}(0)}$, which is a groupoid with 1-connected source fibers ($i_0:\mu^{-1}(0)\hookrightarrow M$ denotes the inclusion). On the other hand, assuming that $(M\red G,L_{M\red G})$ is integrable, it follows that $(\mu^{-1}(0),i_0^*L)$ is also integrated by the groupoid $\pi_0^*\G(L_{M\red G})$, and this groupoid has source connected fibers ($\pi_0:\mu^{-1}(0)\to M\red G$ denotes the projection). It follows that there is an embedded, normal, Lie subgroupoid $\NN\subset \tilde{J}^{-1}(0)$ such that:
\[ \tilde{J}^{-1}(0)/\NN\simeq \pi_0^*\G(L_{M\red G}). \]
Finally, one checks that one has $\NN=\Mt_0$, so that this is an embedded Lie subgroupoid, as claimed.
\end{rem}

\subsection{Obstructions to Integrability of Hamiltonian quotients}      
\label{sub:obstr}                                         

In the previous paragraph we have shown that the bundle of $G$-monodromy groups $\Mt\subset (M,L)$ control the integrability of the Hamiltonian quotient of a proper, free Hamiltonian $G$-space $(M,L,G,\mu)$, when $\mu$ is regular. Notice that although we have not shown yet how to compute these groups, from the fact that they are the image of homomorphisms $\partial_m:\pi_1(G)\to\G(L)_m$, whose image is in the center, we can already conclude that the obstructions to integrability vanish whenever:
\begin{enumerate}
\item[(a)] $\pi_1(G)$ is finite;
\item[(b)] the isotropy groups $\G(L)_m$ have trivial center.
\end{enumerate}

Although the $G$-monodromy morphism can be hard to compute in general, we show now that for points in regular leaves  they have a nice geometric description, in terms of variations of the presymplectic areas of disks whose boundary is an orbit of a closed loop in $G$.

\subsubsection{The restricted $G$-monodromy groups}
The first step is to observe that since $\Mt_m$ is contained in the center of $\G(L)_m$, we can decide about its lack of discreteness by looking at its intersection with the connected component of the center: 
\begin{equation}
\label{eq:conn:monodromy}
 \Mt_m^0:=\Mt_m\cap C(\G(L)_m)^0.
\end{equation}
We have the following characterization of the groups $\Mt_m^0$: denote by $Z(\Ker\sharp_m)$ the center of the isotropy Lie algebra. A element $z\in Z(\Ker\sharp_m)$ determines an $A$-path $a(t)=z$, and we denote its $A$-homotopy class by $[z]\in G(L)$. In fact, the map $Z(\Ker\sharp_m)\to \G(L)_m$, $z\mapsto [z]$, is just the restriction of the exponential map $\exp:\Ker\sharp_m\to \G(L)_m$ to the center of the isotropy Lie algebra. In particular, $[z]\in C(\G(L)_m)^0$, and we have:

\begin{prop}
The groups \eqref{eq:conn:monodromy} are given by:
\[  \Mt_m^0=\bigl\{[a]\in \Mt_m : a\sim z, \text{ for some }z\in Z(\Ker\sharp_m)\bigr\} \]
In particular, if $\partial_m [\gamma]\in  \Mt_m^0$ then $\gamma_m(t):=\gamma(t)m$ is contractible in the presymplectic leaf $S_m$ containing $m$.
\end{prop}
\begin{proof}
Let $[\gamma]\in\pi_1(M)$. Recall that $\partial_m\gamma$ is represented by the $A$-path:
\[ a_1(t)=(\xi(t)_M|_{\gamma(t)\cdot m},\d_{\gamma(t)\cdot m}\mu_{\xi(t)}), \]
where $\xi(t)=\dot{\gamma(t)}\gamma(t)^{-1}\in\gg$ (see \eqref{eq:inner:morphism}).  Note that the base path of $a_1(t)$ is precisely $\gamma_m(t):=\gamma(t)\cdot m$. Assume that $\partial_m\gamma\in C(\G(L)_m)^0$. Then there exists some $z\in Z(\Ker\sharp_m)$ such that $\exp(z)=\partial_m\gamma$, which means that $[a_1]=[a_0]$ where $a_0(t)$ is the $A$-path $a_0(t)=tz$, with base path the constant path $\gamma_0(t)=m$. Since $a_0$ and $a_1$ are $A$-homotopic, the corresponding base paths $\gamma_0$ and $\gamma_m$ are homotopic in the leaf $S_m$, so we have that $\gamma_m$ is contractible in the leaf $S_m$.
\end{proof}

This justifies the following definition:

\begin{defn}
We call \textbf{restricted $G$-monodromy group at $m$} the unique subgroup $\M_m\subset\Ker\sharp_m$ such that:
\[ \M_m=\set{z\in Z(\Ker\sharp_m): [z]\in \Mt_m }. \]
\end{defn}

Since the exponential map maps the center of a Lie algebra onto the connected component of the center of the Lie group, it is clear that $\M_m$ is the unique subgroup of $Z(\Ker\sharp_m)$ such that:
\[ \exp(\M_m)= \Mt_m^0=\Mt_m\cap C(\G(L)_m)^0.\]
In particular, the groups $\M_n$ contain the kernel of the exponential map. This kernel, denoted $\NN_m$, are the monodromy groups that control the integrability of $L$, and can be intrinsically defined by:
\[ \NN_m=\set{z\in Z(\Ker\sharp_m): z\sim 0_m\text{ as $A$-paths}}. \]

\subsubsection{The regular case}
We will show now that at regular presymplectic leaves the restricted monodromy groups have a nice geometric description in terms of transverse variations of presymplectic areas.

Fix a point $m\in M$, let $S$ be the presymplectic leaf through $m$. Given a path $\gamma:I\to G$, based at the identity, such that $\gamma_m(t):=\gamma(t)m$ is contractible in $S$, consider a 2-disk $\Gamma: D^2\to S$, which maps $(1,0)$ to $m$ and maps the boundary $\partial D$ to $\gamma_m$. The presymplectic area of $\Gamma$ is given, as usual, by
\[ A_{\omega}(\Gamma)= \int_{D^2} \Gamma^{*}\omega ,\]
where $\omega$ is the prsymplectic 2-form on the leaf $S$. 

By a \textbf{deformation} of $\Gamma$ we mean a family $\Gamma_{t}: D^2\to M$ of 2-disks parameterized by $t\in (-\eps,\eps)$, starting at
$\Gamma_{0}= \Gamma$, and such that for each fixed $t$ the disk $\Gamma_t$ has image lying entirely in a presymplectic leaf. The
\textbf{transversal variation} of $\Gamma_t$ (at $t=0$) is the class of
the tangent vector
\[
\var_{\nu}(\Gamma_{t})\equiv
\left[\left.\frac{\d}{\d t}\Gamma_{t}(1,0)\right|_{t=0}\right]\in
\nu_{m}(S).
\]
where $\nu_{m}(S)=T_SM/TS$ denotes the normal space to $S$ at $m$. Note that the conormal space $\nu^*_m(S)$ coincides with the isotropy Lie algebra $\Ker\sharp|_m$. The formula
\[ \langle A^{'}_{\omega}(\Gamma), \var_{\nu}(\Gamma_{t})\rangle=
\left.\frac{\d}{\d t}A_{\omega}(\Gamma_{t})\right|_{t=0},\] 
applied to different deformations of $\Gamma$, gives an element
\[ A^{'}_{\omega}(\Gamma)\in (\Ker\sharp|_m)/\NN_m,\]
which, as we shall see, only depends on the homotopy class $[\gamma]\in\pi_1(G)$.
\vskip 15 pt
\begin{center}
\includegraphics[width=.8\textwidth]{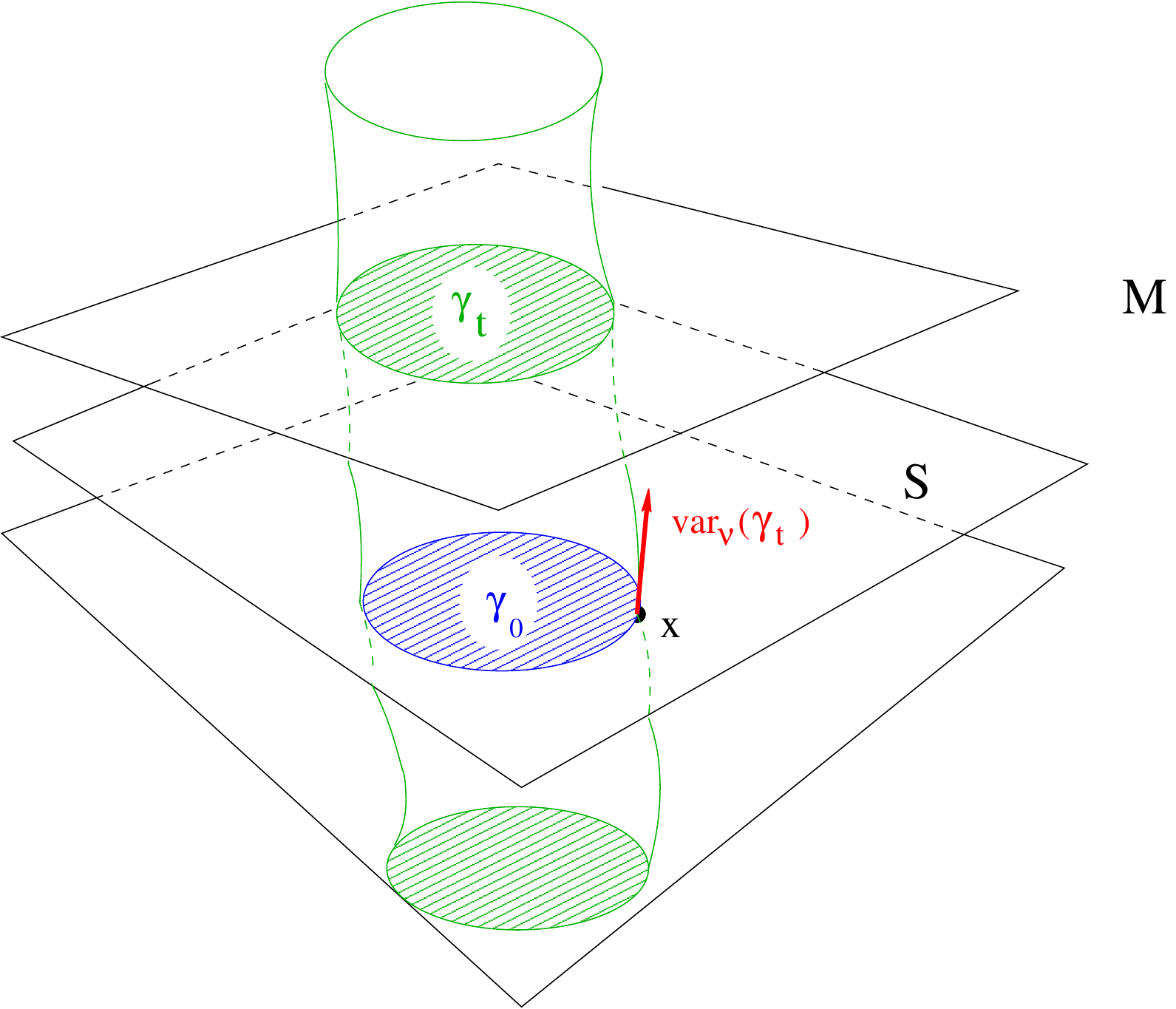}
\end{center}

Now, we have:

\begin{thm}
  \label{var-area}
  For a regular presymplectic leaf $S$ and $m\in S$: 
  \[ \M_{m}= \set{A^{'}_{\omega}(\Gamma): \text{ $\Gamma$ any disk in $S$ with boundary $\gamma(t)\cdot m$ with }[\gamma]\in\pi_1(G)}.\]
\end{thm}

\begin{proof}
Let us recall some standard results about $A$-homotopies (see Section 3 in \cite{CrFe3}). Let $S$ be a regular presymplectic leaf of $(M,L)$. A splitting $\sigma:TS\to L_S$ of the short exact sequence:
\[ 
\xymatrix{0\ar[r]& \gg_S\ar[r]& L_S\ar[r]^{\sharp} & TS\ar[r] \ar@/^1pc/@{-->}[l]^{\sigma}& 0}
\]
induces a connection $\nabla^{\sigma}$ on the bundle $\mathfrak{g}_S\to S$ by setting:
\[ (\nabla^{\sigma}_X v)(m):=[\sigma(X),v(m)]. \]
This splitting has curvature 2-form $\Omega_{\sigma}\in\Omega^2(S;\gg_S)$:
\[ \Omega_{\sigma}(X, Y):= \sigma ([X, Y])- [\sigma(X), \sigma(Y)]\ .\]
Since $S$ is a regular leaf, $\gg_S$ is a bundle of abelian Lie algebras, and it follows that $\nabla^{\sigma}$ is a flat connection:
\[ R_{\nabla^{\sigma}}(X,Y)v=[\Omega_{\sigma}(X,Y),v]=0. \]
We can also use $\sigma$ to identify $L_S$  with $TS\oplus \mathfrak{g}_S$ so the bracket becomes:
\[ [(X, v), (Y, w)]= ([X, Y], [v, w]+ \nabla_{X}^{\sigma}(w)- \nabla_{Y}^{\sigma}(v)- \Omega_{\sigma}(X, Y)) .\]
If we choose some connection $\nabla^{S}$ on $S$ and consider the connection $\nabla=(\nabla^{S},\nabla^{\sigma})$ on $L_S=TS\oplus
  \mathfrak{g}_S$, it follows that:
\[ T_{\nabla}((X, v), (Y, w))= (T_{\nabla^{M}}(X, Y),  \Omega_{\sigma}(X, Y)- [v, w]) \]
for all $X, Y\in TS$, $v, w\in \mathfrak{g}_S$. This shows that for the algebroid $A=L_S$, any $A$-homotopy $a(\epsilon,t)\d t + b(\epsilon,t)\d \epsilon$ takes the form $a=(\frac{\d\gamma}{\d t},\phi)$, $b=(\frac{\d\gamma}{\d\epsilon},\psi)$, where $\phi$ and $\psi$ are paths in $\mathfrak{g}_S$ satisfying:
\begin{equation}
\label{eq:homotopy:2}
  \partial_{t}\psi-\partial_{\epsilon}\phi=\Omega_{\sigma}(\frac{\d\gamma}{\d t},\frac{\d\gamma}{\d\epsilon})-[\phi, \psi].
\end{equation}

Now fix $m\in M$ and let $S$ be the (regular) presymplectic leaf of $(M,L)$ through $m$. Since the $G$-action is free and hamiltonian, we can choose a splitting $\sigma:TS\to L_S$ with the property that $\sigma(\xi_M)=(\xi_M,\d\mu_\xi)$, for any $\xi\in\gg$. Hence, given $[a_0]\in\M_m$, its expression relative to the splitting $\sigma$ takes the form $a_0=(\frac{\d\gamma}{\d t},0)$. Moreover, since $a_0\in\G(L)^0$, the loop $\gamma(t)$ is contractible in the leaf $S$. Denote by $\gamma(t,s)$ a smooth homotopy in the leaf $S$, joining $\gamma(t)$ to the constant path at $m$. If we set $\psi=0$ and let:
\[ \phi(t,\epsilon)=\int_0^\epsilon \Omega(\frac{\d\gamma}{\d t},\frac{\d\gamma}{\d\epsilon})\d \epsilon, \]
Then $a(t,\epsilon)=(\frac{\d\gamma}{\d t}(t,\epsilon),\phi(t,\epsilon))$, $b(t,\epsilon)=(\frac{\d\gamma}{\d\epsilon}(t,\epsilon),0)$ gives a homotopy between $a_0(t)$ and a $\gg_m$-path $v:I\to\gg_m$, with:
\[ v(t)=\int_0^1 \Omega(\frac{\d\gamma}{\d t},\frac{\d\gamma}{\d\epsilon})\d \epsilon. \]
Since $\gg_m$ is abelian, it follows that the class $[a_0]=[v(\cdot)]\in\G_m$ coincides with $\exp(v_0)$, where $v_0\in\gg_m$ is the element:
\[ v_0=\int_0^1 v(t)\d t= \int_{\gamma}\Omega. \]

Finally, an argument analogous to the proof of Proposition 5.4 in \cite{CrFe3} shows that for a family $\Gamma_{t}: D^2\to M$ of 2-disks parameterized by $t\in (-\eps,\eps)$, such that for each fixed $t$ the disk $\Gamma_t$ has image lying entirely in a presymplectic leaf, one has:
 \[ \frac{\d}{\d t} \int_{D^2}\Gamma_{t}^{*}\omega=
  \langle \int_{\Gamma_{t}}\Omega,
  \frac{\d}{\d t}\Gamma_{t}(1,0)\rangle.\] 
which completes the proof of the Theorem.
\end{proof}

\begin{rem}
If $\Gamma_1$ and $\Gamma_2$ are two disks in $S$ with the same boundary $\gamma(t)\cdot m$, then they determine a sphere $\sigma$ and hence a homotopy class $[\sigma]\in\pi_2(S,m)$. The results in \cite{CrFe2} then show that: 
\[ A^{'}_{\omega}(\Gamma_2)-A^{'}_{\omega}(\Gamma_1)\in \NN_m. \]
It follows that $A^{'}_{\omega}(\Gamma)\in (\Ker\sharp|_m)/\NN_m$ only depends on the homotopy class $[\gamma]\in\pi_1(G)$.  In fact, we can think of the map 
\[ \pi_1(G,S_m)\to (\Ker\sharp|_m)/\NN_m,\quad [\gamma]\mapsto A^{'}_{\omega}(\Gamma) \]
where $\pi_1(G,S_m)$ denotes the classes in $\pi_1(G)$ such that $\gamma_m$ is contractible in $S_m$, as the infinitesimal version the $G$-monodromy map.
\end{rem}

\subsubsection{An non-integrable Hamiltonian quotient} Let us give an example of an integrable Hamiltonian $G$-space $(M,L,G,\mu)$, where the action is proper, free and regular, and for which the Hamiltonian quotient $M\red G$ is not integrable. The origins of this example are in the theory of coupling Dirac structures and will be dealt with in a separate publication \cite{BrFe2}.

Let $P\to B$ be a principal $G$-bundle and fix a connection 1-form $\theta:TP\to\gg$. Assume also that $G$ acts on a Poisson manifold $(F,\pi_F)$ in a Hamiltonian fashion with moment map $\mu_F:F\to\gg^*$. The associated bundle $E=P\times_G F$ inherits a Dirac structure which we now recall.

The total space of the co-vertical bundle $S:=(\ker \d q)^*$ is naturally isomorphic as a $G$-space to $P\times \g^*$, with the diagonal $G$-action. 
The  $G$-principal connection $\theta:TP\to \g$, dualizes to give an embedding $i_\theta:S\hookrightarrow T^*P$ and hence determines a presymplectic form $\omega_S=i_\theta^*\omega_\text{can}$ on $S$. One checks easily that the action of $G$ on the Dirac manifold $(S,\omega_S)$ is hamiltonian with moment map $\mu_S:S\to\gg^*$ the projection $S=P\times\g^*\to\g^*$. 

Now let us consider the Dirac manifold $(M,L)=(S\times F , L_{\omega_{S}}\oplus L_{\pi_F})$. This Dirac manifold with the diagonal action of $G$ on $S \times F$, becomes a Hamiltonian $G$-space $(M,L,G,\mu)$ with moment map $\mu=\mu_{S}+\mu_F$. The $G$-action on $M$ is proper and free, since the action on the factor $S$ is proper and free. Also, the moment map $\mu$ is regular, since the presymplectic leaves of $M$ are of the form $S\times S_F$ where $S_F$ is a symplectic leaf of $(F,\pi_F)$ and $\mu_S$ is regular. We conclude from Theorem \ref{thm:Dirac:reduction}, that we have Dirac structures on $M/G$, $\mu^{-1}(0)$ and $M\red G\simeq P\times_G F=E$. 

Assume now that $(F,\pi_F)$ is an integrable Poisson manifold. It follows that the Dirac structure $L$ on $M=S\times F$ is also integrable, since $(S,\omega_S)$ is always integrable (see Example \ref{ex:presymplectic:int}) and the product of integrable Dirac structures is an integrable Dirac structure. Theorem \ref{thm:int:Dirac:quotients} then says that $M/G$ is also an integrable Dirac structure, but the following example shows that $M\red G$ may fail to integrable.

\begin{ex}
Let us consider the special case where $G=\Ss^1$, so that $\gg=\Rr$, $\pi_F=0$, and $\Ss^1$ acts trivially on $F$. Then we have that:
\[ M\red G=B\times F.\]
Denoting by $\omega_\theta\in \Omega^2(B,\Rr)$ the curvature 2-form of the connection, one sees that the Dirac structure $L_{M/G}$ has presymplectic foliation with leaves $S_f:=B\times {f}$, $f\in F$, where the presymplectic form is given by:
\[ \omega_{S_f}=\mu_F(f)\omega_\theta. \]
If the curvature form is ``fat'', i.e., is non vanishing, and $\mu_F$ is also non-vanishing, this is actually a Poisson structure.

For a specific example, let $p:\Ss^3\to \Ss^2$ be the Hopf fibration with a connection 1-form $\theta$ such that its curvature form $\omega_\theta\in\Omega^2(\Ss^2,\Rr)$ is nowhere vanishing and let $F=\Rr$. If $\mu_F:\Rr\to\Rr$ is also a non-vanishing function $\mu_F(r)=f(r)$, then $M\red G$ is the Poisson manifold $\Ss^2\times\Rr$ with symplectic leaves the spheres $S_r:=\Ss^2\times \{r\}$, $r\in\Rr$, where the symplectic form is given by:
\[ \omega_{S_r}=f(r)\omega_\theta. \]
This Poisson structure is not integrable around leaves $S_r$ where the function $f(r)$ has critical points (see \cite{CrFe1}). 
\end{ex}

This example is part of a general theory of geometric integration of Yang-Mills phase spaces, and more general coupling Dirac structures, which will appear in a separate paper \cite{BrFe2}

\bibliographystyle{amsalpha}


\end{document}